\documentclass[a4paper,12pt]{amsart}

\usepackage{fancyhdr}
\usepackage{a4wide}
\usepackage{tikz}
\usepackage{amsmath,amsthm}
\usepackage{float}
\usepackage{color}
\usepackage[noadjust]{cite}
\usepackage{enumitem}
\usepackage{cleveref}
\newcommand{\mylabel}[2]{#2\def\@currentlabel{#2}\label{#1}}
\DeclareMathOperator{\de}{deg}
\DeclareMathOperator{\SO}{SO}
\newcommand{\suv}{\sum_{ uv \in E(G)}}

\newcommand{\LE}{\vartriangleleft}

\newcommand{\Mi}{\mathcal{M}}
\DeclareMathOperator{\rr}{r}

\newcommand{\changefont}{%
    \fontsize{9}{12}\selectfont}

\newtheorem{thm}{Theorem}[section]

\newtheorem{lem}[thm]{Lemma}
\newtheorem{cor}[thm]{Corollary}
\theoremstyle{remark}
\newtheorem{rem}[thm]{Remark}
\newtheorem{defn}[thm]{Definition}

\newtheorem*{thm*}{Theorem}

\newif\ifdetails
\detailstrue
\newcommand{\DETAIL}[1]%
{\ifdetails\par\fbox{\begin{minipage}{0.9\linewidth}\textit{Detail:}
      #1\end{minipage}}\par\fi}
\newcommand{\TODO}[1]%
{\ifdetails\par\fbox{\begin{minipage}{0.9\linewidth}\textbf{TODO:}
      #1\end{minipage}}\par\fi}

\newcommand{\old}[1]{{}}

\title{On the Sombor index of trees with degree restrictions}

\author{Eric O. D. Andriantiana}
\address{Eric O. D. Andriantiana\\
Department of Mathematics (Pure and Applied)\\
Rhodes University, PO Box 94\\
6140 Grahamstown\\
South Africa}
\email{E.Andriantiana@ru.ac.za}

\author{Valisoa R. M. Rakotonarivo}
\address{Valisoa R. M. Rakotonarivo\\
Department of Mathematics and Applied Mathematics\\
University of Pretoria\\
Hatfield 0028\\
South Africa}
\email{valisoa.rakotonarivo@up.ac.za}


\begin{document}

\begin{abstract}
We study the Sombor index of trees with various degree restrictions. In addition to rediscovering that among all trees with a given degree sequence, the greedy tree minimises the Sombor index and the alternating greedy tree maximises it, we also provide a full characterisation of all trees that have those maximum or minimum values. Moreover, we compare trees with different degree sequences and deduce a few corollaries. 
\end{abstract}

\maketitle

\section{Introduction}
The Sombor index finds its roots in the exploration of molecular graphs, which represent the structure of chemical compounds as graphs with atoms as vertices and bonds as edges. Ivan Gutman, a pioneering figure in chemical graph theory, introduced this novel invariant in 2021 \cite{gutman2021}, which derives from the vertex degrees. 

All graphs considered in this paper are simple, finite and undirected. Let $G$ be a graph, with vertex set $V(G)$ and edge set $E(G)$. The Sombor index, denoted by $\SO(G)$ is defined as \begin{equation}
\SO(G)=\suv \sqrt{\left(\de_G(u)^2+\de_G(v)^2\right)},
\end{equation}
where $\de_G(u)$ is the degree of the vertex $u \in V(G)$, i.e. the number of edges adjacent to $u$. In a clear context, we may write $\de(u)$.

A generalized version, the $\alpha$-Sombor index of $G$ was defined in \cite{Tam2021} as follows:
\begin{equation}
\SO_{\alpha}(G)=\suv \left(\de_G(u)^2+\de_G(v)^2\right)^{\alpha}.
\end{equation}
The original Sombor index corresponds to $\alpha=1/2$.

This topological index holds significant implications to simulate some chemical properties of compounds. It provides valuable insights into their physical and chemical properties.

From a mathematical perspective, this invariant was widely popular in the recent couple of years. Extensive research has been conducted to investigate the behaviour of the Sombor index within the class of trees with various restrictions as seen in \cite{li2022,chen2022,damnjanovic2022} and for general graphs as well \cite{Wei2022,Tam2021}. 

In this paper, we focus on the set of all trees with given degree sequence $D=(d_1,d_2,$ $\dots,d_n)$, which we will denote by $\mathbb{T}_D$. It is not surprising that the extremal trees regarding the Sombor index are the well-known greedy trees and alternating greedy trees as proved in \cite{damnjanovic2022} and \cite{Wei2022}. However, they may not be unique depending on the sequence $D$, and our approach will indeed characterize all the extremal trees in $\mathbb{T}_D$. In addition, we will compare graphs with different degree sequences to deduce a few results related to other restrictions such as a given maximum degree, diameter, number of leaves, and number of branching points.


\section{Preliminaries}

The following lemma describes how the Sombor index is affected by a swap of two branches. It will play a central role to the rest of the paper.
\begin{lem}\label{switching}
Let $T \in \mathbb{T}_D$. Consider $T$ as a (vertex/edge) rooted tree. Let $u,v \in V(T)$ such that $\de(u) < \de(v)$ and $z,w$ their respective children such that $\de(w)<\de(z)$. Let $T'=T-vw-uz+vz+uw$. Then $\SO(T')<\SO(T)$. 
\end{lem}

\begin{proof}
Note that the transformation from $T$ to $T'$ preserves the degree sequence, and that the contribution to the Sombor index from any other edges than $vw, uz, vz,uw$  will not be affected. Hence, we are left to show that
\begin{align*}\label{diff}
\SO(T')-\SO(T)&=\sqrt{\de(v)^2+\de(z)^2}+\sqrt{\de(u)^2+\de(w)^2}\\
&-\sqrt{\de(v)^2+\de(w)^2}-\sqrt{\de(w)^2+\de(z)^2} <0.
\end{align*}
From the assumption, we have
$$
(\de(v)^2-\de(u)^2)(\de(w)^2-\de(z)^2)<0,
$$
which expands to
\begin{align*}
&\de(v)^2\de(w)^2+\de(z)^2\de(u)^2-\de(v)^2\de(z)^2-\de(w)^2\de(u)^2\\
&= (\de(v)^2+\de(z)^2)(\de(u)^2+\de(w)^2)-(\de(v)^2+\de(w)^2)(\de(u)^2+\de(z)^2)<0.
\end{align*} 
This is equivalent to 
\begin{align*}
&\sqrt{(\de(v)^2+\de(z)^2)(\de(u)^2+\de(w)^2)}-\sqrt{(\de(v)^2+\de(w)^2)(\de(u)^2+\de(z)^2)}\\
&=\frac{1}{2}\left[\left(\sqrt{\de(v)^2+\de(z)^2}+\sqrt{\de(u)^2+\de(w)^2}\right)^2\right.\\
&\hspace{5.5cm} \left.-\left(\sqrt{\de(v)^2+\de(w)^2}+\sqrt{\de(w)^2+\de(z)^2}\right)^2\right]<0.
\end{align*}
and
\begin{align*}
&\sqrt{\de(v)^2+\de(z)^2}+\sqrt{\de(u)^2+\de(w)^2}-\sqrt{\de(v)^2+\de(w)^2}-\sqrt{\de(w)^2+\de(z)^2}\\&<0.
\end{align*}
\end{proof}

\begin{rem}\label{same}
If, in Lemma \ref{switching}, we allow $\de(u)=\de(v)$ or $\de(w)=\de(z)$ , then the Sombor index remains the same.
\end{rem}

\section{Trees of order $n$ with the same degree sequence}

Let $D=(d_1,d_2,\dots,d_n)$ be a degree sequence of an $n$-vertex tree. In this section, we are studying the Sombor index of trees in $T \in \mathbb{T}_D$. Wei and Liu \cite{Wei2022} compared the Sombor indices of graphs with given degree sequence. The fact that $G(D)$ and $\mathcal{M}(D)$, defined below, are extremal graphs among all elements of $\mathbb{T}_D$ is a special case of their findings. Our main contribution is not only to provide a simple proof of the extremality of these two graphs, but to characterise all of the extremal graphs regarding the Sombor index.

For completeness, we recall the following popular definitions of greedy trees and level greedy trees, see \cite{schm12,wang08,zhang08b}.

\begin{defn}
Let $T$ be a rooted tree, where the maximum distance from the root is $k-1$. The leveled degree sequence of $T$ is the sequence 
\begin{equation}
\label{Eq:D}
D=(V_1, \dots,V_k),
\end{equation}
where, for any $1\leq i \leq k$, $V_i$ is the non-increasing sequence formed by the degrees of the 
vertices of $T$ at the $i^{\text{th}}$ level (i.e., vertices of distance $i-1$ from the root in the respective component). 
\end{defn}
\begin{defn}
\label{Def:lgf}
The \emph{level greedy tree} with leveled degree sequence 
\begin{equation}\label{eq:ldegseq}
D=((i_{1,1},\dots,i_{1,k_1}), (i_{2,1},\dots,i_{2,k_2}),\dots,(i_{n,1},\dots,i_{n,k_n})),
\end{equation}
where $1\leq k_1 \leq 2$,
is obtained using the following ``greedy algorithm'': 
\begin{enumerate}
\item[(i)] Label the vertices of the first level by $g_1^1,\dots,g^1_{k_1}$, and assign degrees $i_{1,1},\dots,i_{1,k_1}$ 
to them. Add the edge $g^1_1g^1_2$ if $k=2$.

\item[(ii)] Assume that the vertices of the $h^{\text{th}}$ level have been labeled 
$g_1^h,\dots,g_{k_h}^h$ and a degree has been assigned to each of them. Then for all 
$1\leq j \leq k_h$ label the neighbors of $g^h_j$ at the $(h+1)^{\text{th}}$ level, if any, 
by $$g_{1+\sum_{m=1}^{j-1}(i_{h,m}-1)}^{h+1},\dots,g^{h+1}_{\sum_{m=1}^{j}(i_{h,m}-1)},$$ and assign 
degrees to the newly labeled vertices such that $\deg g_j^{h+1} = i_{h+1,j}$ for all $j$.
Unless $h=1$, in which case all the $g^{h+1}_j$s are set to be neighbors of $g^h_1.$
\end{enumerate}
The level greedy tree with leveled degree sequence $D$ is denoted by $G(D)$.
\end{defn}

\begin{defn}
\label{Def:gf}
If a root in a tree can be chosen such that it becomes a level greedy tree whose 
leveled degree sequence, as given in \eqref{eq:ldegseq},
satisfies 
$$\min (i_{j,1},\dots,i_{j,k_j})\geq \max (i_{j+1,1},\dots,i_{j+1,k_{j+1}})$$
for all $1\leq j\leq n-1$, then it is called a \emph{greedy tree}. 
 Even in the case that $D$ is a degree sequence (instead of a leveled degree sequence), we keep the notation $G(D)$ to denote the greedy tree with degree sequence $D$.
\end{defn}

We also need to recall the following definition of alternating greedy trees \cite{andriantiana2013}, and we will introduce the notion of level alternating greedy trees.

Let $T_1,\dots,T_k$ be rooted trees. Then $[T_1,\dots,T_k]$ is a tree with root $v$, obtained by adding an edge to join $v$ to the root of $T_i$ for all $i$. 
\begin{defn}
A complete branch $T_u$ of a rooted tree $T$ is the induced subgraph of $T$ spanned by $u$ and all its descendant.

A complete branch $B=[B_1,\dots,B_k]$ of a tree $T$ is a \textit{pseudo-leaf} if $|V(B_1)|=|V(B_2)|=\cdots=|V(B_k)|=1$: the branches attached to the root $\rr(B)$ of $B$ are all leaves. 
\end{defn}

We simply write $[d]$ for a pseudo-leaf branch with $d$ vertices, and hence has a root of degree $d-1$.

\begin{defn}
Let $(d_1,\dots,d_t,1,\dots,1)$ be the degree sequence of a tree $T$, where $d_j \geq 2$ for $1\leq j \leq t$. The $t$-tuple $(d_1,\dots,d_t)$ is called the \textit{reduced degree sequence} of $T$.
\end{defn}

No information is lost by passing from a degree sequence of a tree to its reduced degree sequence. Using the Handshake Lemma, the number of leaves in a tree $T$ with reduced degree sequence $(d_1,\dots,d_t)$ can be recovered as
$$
k=-2t+2+\sum_{i=1}^td_i.
$$
Two trees with the same reduced degree sequence have the same number of leaves, therefore they have the same degree sequence. 

\begin{defn}
Let $(d_1,\dots,d_t)$ be a reduced degree sequence of a tree. If $t\leq d_t+1$, then $\Mi(d_1,\dots,d_t)$ is the tree obtained by merging the root of each of $[d_1],\dots,[d_{t-1}]$ with a leaf of $[1+d_t]$, respectively. We label all non-leaf vertices as shown in Figure \ref{fig2}, in such a way that
\begin{equation}
\de(v_i) \leq \de(v_j)\quad \text{if }\, i <j.\label{eq4}
\end{equation}

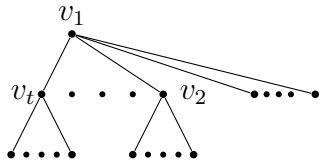
\begin{figure}[H]
$$
\begin{tikzpicture}[scale=0.8]
\node[fill=black,circle,inner sep=1pt] (v1) at (7,0) {};
\node[fill=black,circle,inner sep=1pt] (v2) at (6.5,-1) {};
\node[fill=black,circle,inner sep=1pt] (v3) at (8.5,-1) {};
\node[fill=black,circle,inner sep=1pt] (v4) at (10,-1) {};
\node[fill=black,circle,inner sep=1pt] (v5) at (11,-1) {};
\node[fill=black,circle,inner sep=1pt] (v6) at (6,-2) {};
\node[fill=black,circle,inner sep=1pt] (v7) at (7,-2) {};
\node[fill=black,circle,inner sep=1pt] (v8) at (8,-2) {};
\node[fill=black,circle,inner sep=1pt] (v9) at (9,-2) {};
\draw (v1)--(v2);
\draw (v1)--(v3);
\draw (v1)--(v4);
\draw (v1)--(v5);
\draw (v2)--(v6);
\draw (v2)--(v7);
\draw (v3)--(v8);
\draw (v3)--(v9);
\node[fill=black,circle,inner sep=0.8pt]  at (7,-1) {};
\node[fill=black,circle,inner sep=0.8pt]  at (7.5,-1) {};
\node[fill=black,circle,inner sep=0.8pt]  at (8,-1) {};
\node[fill=black,circle,inner sep=0.8pt]  at (10.2,-1) {};
\node[fill=black,circle,inner sep=0.8pt]  at (10.4,-1) {};
\node[fill=black,circle,inner sep=0.8pt]  at (10.6,-1) {};
\node[fill=black,circle,inner sep=0.8pt]  at (6.25,-2) {};
\node[fill=black,circle,inner sep=0.8pt]  at (6.5,-2) {};
\node[fill=black,circle,inner sep=0.8pt]  at (6.75,-2) {};
\node[fill=black,circle,inner sep=0.8pt]  at (8.25,-2) {};
\node[fill=black,circle,inner sep=0.8pt]  at (8.5,-2) {};
\node[fill=black,circle,inner sep=0.8pt]  at (8.75,-2) {};
\node at (7,0.3) {$v_1$};
\node at (6.2,-1) {$v_t$};
\node at (9,-1) {$v_2$};
\end{tikzpicture}
$$
\caption{Labelling of the vertices of $\Mi(d_1,\dots,d_t)$ when $t\leq d_t+1$.}
\label{fig2}
\end{figure}

On the other hand, if $t \geq d_t +2$, we construct $\Mi(d_1,\dots,d_t)$ recursively: let $\ell$ be the greatest integer such that $v_{\ell}$ is a label in $\Mi(d_{d_t},\dots,d_{t-1})$. Let $s$ be the smallest integer such that $v_s$ is adjacent to a leaf in $\Mi(d_{d_t},\dots,d_{t-1})$. Let $R_{d_t}=[[d_1],\dots,[d_{d_t-1}]]$, where the pseudo-leaves are labelled $v_{\ell+1},\dots,$ 
$v_{\ell+d_t-1}$ still respecting \eqref{eq4}. $\Mi(d_1,\dots,d_t)$ is the tree obtained by merging the root of $R_{d_t}$ to a leaf adjacent to $v_s$.

\begin{figure}[H]
$$
\begin{tikzpicture}[scale=0.7]
\node[fill=black,circle,inner sep=1pt] (t1) at (0,0) {};
\node[fill=black,circle,inner sep=1pt] (t2) at (-2,-1) {};
\node[fill=black,circle,inner sep=1pt] (t3) at (0,-1) {};
\node[fill=black,circle,inner sep=1pt] (t4) at (2,-2) {};
\node[fill=black,circle,inner sep=1pt] (t5) at (-2.5,-2) {};
\node[fill=black,circle,inner sep=1pt] (t6) at (-2,-2) {};
\node[fill=black,circle,inner sep=1pt] (t7) at (-1.5,-2) {};
\node[fill=black,circle,inner sep=1pt] (t8) at (-0.5,-2) {};
\node[fill=black,circle,inner sep=1pt] (t9) at (0.5,-2) {};
\draw (t1)--(t2);
\draw (t1)--(t3);
\draw (t1)--(t4);
\draw (t2)--(t5);
\draw (t2)--(t6);
\draw (t2)--(t7);
\draw (t3)--(t8);
\draw (t3)--(t9);
\node at (0,0.3) {$v_1$};
\node at (-2.3,-1) {$v_3$};
\node at (0.3,-1) {$v_2$};
\draw[->, line width=1pt] (2,-1)--(2.6,-1);
\node[fill=black,circle,inner sep=1pt] (s1) at (5.5,0) {};
\node[fill=black,circle,inner sep=1pt] (s2) at (3.5,-1) {};
\node[fill=black,circle,inner sep=1pt] (s3) at (5.5,-1) {};
\node[fill=black,circle,inner sep=1pt] (s4) at (7.5,-2) {};
\node[fill=black,circle,inner sep=1pt] (s5) at (3,-2) {};
\node[fill=black,circle,inner sep=1pt] (s6) at (3.5,-2) {};
\node[fill=black,circle,inner sep=1pt] (s7) at (4,-2) {};
\node[fill=black,circle,inner sep=1pt] (s8) at (5,-2) {};
\node[fill=black,circle,inner sep=1pt] (s9) at (6,-2) {};
\node[fill=black,circle,inner sep=1pt] (s10) at (7,-3) {};
\node[fill=black,circle,inner sep=1pt] (s11) at (8,-3) {};
\node[fill=black,circle,inner sep=1pt] (s12) at (6.7,-4) {};
\node[fill=black,circle,inner sep=1pt] (s13) at (7,-4) {};
\node[fill=black,circle,inner sep=1pt] (s14) at (7.3,-4) {};
\node[fill=black,circle,inner sep=1pt] (s15) at (7.7,-4) {};
\node[fill=black,circle,inner sep=1pt] (s16) at (8,-4) {};
\node[fill=black,circle,inner sep=1pt] (s17) at (8.3,-4) {};
\draw (s1)--(s2);
\draw (s1)--(s3);
\draw (s1)--(s4);
\draw (s2)--(s5);
\draw (s2)--(s6);
\draw (s2)--(s7);
\draw (s3)--(s8);
\draw (s3)--(s9);
\draw (s4)--(s10);
\draw (s4)--(s11);
\draw (s10)--(s12);
\draw (s10)--(s13);
\draw (s10)--(s14);
\draw (s11)--(s15);
\draw (s11)--(s16);
\draw (s11)--(s17);
\node at (5.5,0.3) {$v_1$};
\node at (3.2,-1) {$v_3$};
\node at (5.8,-1) {$v_2$};
\node at (8.3,-3) {$v_4$};
\node at (6.7,-3) {$v_5$};
\draw[->, line width=1pt] (7.5,-1)--(8,-1);
\node[fill=black,circle,inner sep=1pt] (v1) at (11,0) {};
\node[fill=black,circle,inner sep=1pt] (v2) at (9,-1) {};
\node[fill=black,circle,inner sep=1pt] (v3) at (11,-1) {};
\node[fill=black,circle,inner sep=1pt] (v4) at (13.5,-2) {};
\node[fill=black,circle,inner sep=1pt] (v5) at (8.5,-2) {};
\node[fill=black,circle,inner sep=1pt] (v6) at (9,-2) {};
\node[fill=black,circle,inner sep=1pt] (v7) at (9.5,-2) {};
\node[fill=black,circle,inner sep=1pt] (v8) at (10.5,-2) {};
\node[fill=black,circle,inner sep=1pt] (v9) at (11.5,-2) {};
\node[fill=black,circle,inner sep=1pt] (v10) at (13,-3) {};
\node[fill=black,circle,inner sep=1pt] (v11) at (14,-3) {};
\node[fill=black,circle,inner sep=1pt] (v12) at (12.7,-4) {};
\node[fill=black,circle,inner sep=1pt] (v13) at (13,-4) {};
\node[fill=black,circle,inner sep=1pt] (v14) at (13.3,-4) {};
\node[fill=black,circle,inner sep=1pt] (v15) at (13.7,-4) {};
\node[fill=black,circle,inner sep=1pt] (v16) at (14,-4) {};
\node[fill=black,circle,inner sep=1pt] (v17) at (14.3,-4) {};
\node[fill=black,circle,inner sep=1pt] (v18) at (11.5,-3) {};
\node[fill=black,circle,inner sep=1pt] (v19) at (11,-4) {};
\node[fill=black,circle,inner sep=1pt] (v20) at (11.3,-4) {};
\node[fill=black,circle,inner sep=1pt] (v21) at (11.6,-4) {};
\node[fill=black,circle,inner sep=1pt] (v22) at (11.9,-4) {};
\draw (v1)--(v2);
\draw (v1)--(v3);
\draw (v1)--(v4);
\draw (v2)--(v5);
\draw (v2)--(v6);
\draw (v2)--(v7);
\draw (v3)--(v8);
\draw (v3)--(v9);
\draw (v4)--(v10);
\draw (v4)--(v11);
\draw (v10)--(v12);
\draw (v10)--(v13);
\draw (v10)--(v14);
\draw (v11)--(v15);
\draw (v11)--(v16);
\draw (v11)--(v17);
\draw (v9)--(v18);
\draw (v18)--(v19);
\draw (v18)--(v20);
\draw (v18)--(v21);
\draw (v18)--(v22);
\node at (11,0.3) {$v_1$};
\node at (8.7,-1) {$v_3$};
\node at (11.3,-1) {$v_2$};
\node at (14.3,-3) {$v_4$};
\node at (12.7,-3) {$v_5$};
\node at (11.2,-3) {$v_6$};
\end{tikzpicture}
$$
\caption{Construction of $\Mi(5,4,4,4,3,3,3,2)$.}
\label{cha5.fig3}
\end{figure}
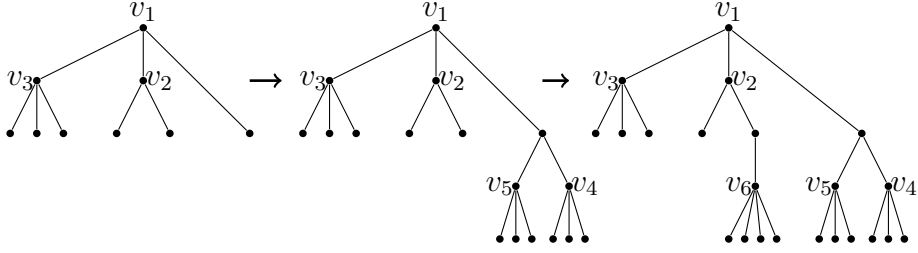
\end{defn}

\begin{defn}
\label{Def:Algf}
The \emph{alternating level greedy tree} with leveled degree sequence 
\begin{equation}\label{eq:aldegseq}
D=((i_{1,1},\dots,i_{1,k_1}), (i_{2,1},\dots,i_{2,k_2}),\dots,(i_{n,1},\dots,i_{n,k_n})), 
\end{equation}
where $1\leq k_1 \leq 2$,
is obtained using the following algorithm: 
\begin{enumerate}
\item[(i)] Label the vertices of the first level by $g_1^1,\dots,g_{k_1}^1$, and assign degrees 
to these vertices such that $\deg g_j^1 = i_{1,j}$ for all $j$.
Add an edge $g^1_1g^1_2$ if $k_1=2$.

\item[(ii)] Assume that the vertices of the $h^{\text{th}}$ level have been labeled 
$g_1^h,\dots,g_{k_h}^h$ and a degree has been assigned to each of them. Then for all 
$1\leq j \leq k_h$ label the neighbors of $g^h_j$ at the $(h+1)^{\text{th}}$ level, if any, 
by 
$$
g_{1+\sum_{m=1}^{j-1}(\deg(g^h_m)-1)}^{h+1},\dots,g^{h+1}_{\sum_{m=1}^{j}(\deg(g^h_m)-1)},$$ and assign 
degrees to the newly labeled vertices such that $\deg g_j^{h+1} = i_{h+1,j}$ for all $j$ if $h$ is even and $\deg g_j^{h+1} = 
i_{h+1,k_{h+1}-j+1}$ if $h$ is odd.

Unless $h=1$, in which case all the $g^{h+1}_j$s are set to be neighbors of $g^h_1.$
\end{enumerate}
\end{defn}
As we pass from one level to the next, in an alternating level greedy tree, edges are being added between vertices of largest degree available in level $h$ to vertices of smallest degree available in level $h+1$.

\subsection{Elements of $\mathbb{T}_D$ with maximum Sombor index}
This subsection provides a full characterisation of all trees of degree sequence $D$ that has the maximum value of the Sombor index.


We say a tree to be a ``maximal" tree if it maximises the Sombor index among all trees in $\mathbb{T}_D$. If $v$ is a vertex of a rooted tree $T$, then $T_v$ is the subtree of $T$ induced by $v$ and all its descendants.

\begin{lem} \label{reverse}
Let $T \in \mathbb{T}_D$. If $T$ is maximal, then the following holds for any level number $i$ and any choice of vertex or edge root $r$ of $T$. 
Let $v_1,\dots,v_k$ be vertices at level $i$ and $u_m^1,\dots,u_m^{d_{v_m-1}}$ the respective children of $v_m$. 
If, for some $s,r\in\{1,\dots,k\}$, we have $ \de(v_p) > \de(v_r)$, then 
\begin{equation}
\label{Eq:vhvt}
\max_{j} \de({u_{p}}^j) \leq \min_{j}\de({u_{r}}^j).
\end{equation}
\end{lem}

\begin{proof}
Suppose there are vertices $v_p$ and $v_r$ at the same level in $T$ such that the inequality \eqref{Eq:vhvt} does not hold. 
We can use the operation in Lemma \ref{switching} to increase the Sombor index, by swapping $T_{u_p^j}$ with $T_{u_{r}^{j'}}$ for some $j$ and $j'$.
\end{proof}
Lemma \ref{reverse} does not impose any condition for the case where $\deg(v_p)=\deg(v_r)$. Exchanging branches between the two vertices would not affect the Sombor index.
Depending on $D$ it is possible to have non-isomorphic elements of $\mathbb{T}_D$ that both satisfy the properties described in Lemma \ref{reverse}. We prove that they are all maximal. We first provide one extremal tree.

\begin{lem}
\label{Lem:MDLem8}
For any reduced degree sequence $D$, the graph $\mathcal{M}(D)$ satisfies the properties described in Lemma \ref{reverse}.
\end{lem}
\begin{proof}
Let $D=(d_1,\dots,d_t)$ be a reduced degree sequence. 
The case of $t\leq d_t+1$ is trivial: $\mathcal{M}(D)$ has the smallest internal degree in the center and any other vertex is either a  pseudo-leaf or leaf. 

Suppose that it holds for all $t\leq \ell$ for some $\ell\geq d_t+1$. And now consider the case of $t=\ell+1$. By the induction assumption, $\mathcal{M}(d_{d_t},\dots,d_{t-1})$ satisfies the property. By Definition \ref{Def:Algf}, $\mathcal{M}(D)$ is obtained from $\mathcal{M}'=\mathcal{M}(d_{d_t},\dots,d_{t-1})$ by merging the root of $R_{d_t}=[d_1,\dots,d_{d_t-1}]$ to a leaf $z$ attached to a vertex $v_s$.
For the rest of the proof, we use the notations from Lemma \ref{reverse} for $\mathcal{M}(D)$. 
The condition is satisfied trivially if $v_r$ is a leaf. So, we assume that $v_r$ is not a leaf. 
If $v_p$ is a pseudo-leaf, then ${\displaystyle \max_{j} \deg(u_s^j)=1}$ and the condition holds again. So, we can also assume that $v_p$ is not a pseudo leaf. 

The vertex $v_p$ cannot be $z$, as $z$ has the smallest internal degree. If $v_r$ is $z$, then, the property is satisfied again, as the children of $z$ has the largest available degrees. Hence, the case involving $z$ is sorted. From now, we assume that $v_p$ and $v_r$ are in $\mathcal{M}'$. 

If $v_s\notin \{v_p,v_r\}$, then the claim holds by induction assumption. 

The minimum degree of the neighbors of $v_s$ does not decrease from $\mathcal{M}'$ to $\mathcal{M}_D$. Hence, if $v_r=v_s$, the claim also holds by the induction assumption.

Lastly, consider the case where $v_p=v_s$. Since the neighbor $z$ of $v_s$ has the smallest internal degree, ${\displaystyle \max_j \deg(u_p^j)}$ only increase (from $1$ to $d_t$) as we pass from $\mathcal{M}'$ to $\mathcal{M}(D)$ if $v_p$ was a pseudo-leaf in $\mathcal{M}'$. If $v_r$ has no leaf child, then the condition is still satisfied. If $v_r$ has a leaf child, then the choice of $s$ in Definition \ref{Def:Algf} is contradicted, because the fact that $s$ is the smallest index such that $v_s$ has a child leaf in $\mathcal{M}'$ also means that $v_s$ has the smallest degree among all vertices adjacent to a leaf.

\end{proof}

We now give a full characterisation of all maximal trees. It will help to understand the proof of Theorem \ref{alter}, to know that $\Mi(d_2,\dots,d_{t-1})$ can be obtained from $\Mi(d_1,\dots,d_{t})$ by chopping out pseudo-leaf of degree $d_1$ attached at a vertex of degree $d_t$.

\begin{thm}\label{alter}
Let $T \in \mathbb{T}_D$. $T$ is maximal if and only if it is isomorphic to $\Mi(D)$, with a root preserving isomorphism, up to iterative exchange of branches between vertices of the same degree.

\end{thm}
\begin{proof}
What we are proving is that all element of $\mathbb{T}_D$ that satisfies the properties in Lemma \ref{reverse} have the same Sombor index, which is thus the maximum possible Sombor index in $\mathbb{T}_D$.

We know by Lemma \ref{Lem:MDLem8} that $\Mi(D)$ satisfies Lemma \ref{reverse}.

Let the degree sequence be $D=(d_1,\dots,d_t,1,\dots,1)$ for some $d_t>1$. We use an induction on $t$. The case of $t=1$ is trivial: $\mathbb{T}_D=\{\mathcal{M}(D)\}$ in this case. Suppose that the claim holds for $t=\ell$. Now, let $t=\ell+1$, and $T$ a maximal tree in $\mathbb{T}_D$. By Lemma \ref{reverse}, $T$ has a pseudo leaf $v$ of degree $d_1$ attached to a vertex $u$ of degree $d_t$. Let $T^-$ be the tree obtained from $T$ by removing $v$ and all leaves attached to it. Then, $T^-$ still satisfies Lemma \ref{reverse}, and have the degree sequence $D^-=(d_2,\dots,d_t-1,1\dots,1)$. By the induction assumption, it is isomorphic to $\Mi(D^-)$, up to permutation of branches of vertices of the same degree. Suppose that after $m$ swaps of branches between vertices of the same degree in $T^-$, we obtain $T^-_1$ such that $f:T^-_1\longrightarrow \mathcal{M}(D^-)$ is an isomorphism. Then, $f(u)$ is either a leaf (if $d_t=2$) or the only vertex of $\mathcal{M}(D^-)$ with degree $d_t-1$ (if $d_t>2$).

Let $G$ be the graph obtained from $\mathcal{M}(D^-)$ by adding an edge to join the root of $[d_1]$ to $f(u)$. If $G$ is isomorphic to $\mathcal{M}(D)$, then $f$ extends easily to an isomorphism between $T$ and $\mathcal{M}(D)$. The only case where $G$ and $\mathcal{M}(D)$ can possibly be not isomorphic is if $d_t-1=1$, and $\mathcal{M}(D)$ is obtained from $\mathcal{M}(D^-)$ by adding an edge to join the root of $[d_1]$ to a vertex $f(\nu)$ that is not $f(u)$. Let $u'$ and $\nu'$ be the parents of $u$ and $\nu$, respectively.

If $u'$ and $\nu'$ have the same degree, then we just swap $T_u$ and $T_{\nu}$ to get a graph isomorphic to $\mathcal{M}(D)$ and still have the same Sombor index.

Now suppose that $\deg(u')>\deg(\nu')$. Since $T$ satisfies Lemma \ref{reverse}, the maximum degree of the child of $u'$ has to be at most the minimum of that of $\nu'$, which is $1$. But this contradicts that $u'$ has $u$ as a child of degree $2$ in $T$.

Similarly, if $\deg(u')<\deg(\nu)$, then it contradicts the fact that $\mathcal{M}(D)$ satisfies Lemma \ref{reverse}.
\end{proof}

\begin{rem}
In view of Remark \ref{same}, depending on the degree sequence, the tree that maximises the Sombor index may not be unique as seen in Figure \ref{notunique}.
\end{rem}

\begin{figure}[H]
\begin{tikzpicture}[scale=0.5]
\tikzstyle{every node}=[draw,circle,fill=black,minimum size=5pt,
                            inner sep=0pt]
\node {}[sibling distance=35mm]
    child {  node {}  [sibling distance=20mm]
    					child {node {}
    						child{node{}} 
    						child{node{}}
    					} }
    child {  node{} [sibling distance=20mm]
                    child { node {} }
                    child{node{}
                    }
             }
    child {  node{} [sibling distance=20mm]
                    child { node {}}
                    child{node{}}
             } ;
\end{tikzpicture}
\hspace{2.5cm}
\begin{tikzpicture}[scale=0.5]
\tikzstyle{every node}=[draw,circle,fill=black,minimum size=5pt,
                            inner sep=0pt]
\node {}[sibling distance=35mm]
    child {  node {}  [sibling distance=20mm]
    					child {node {}
    						child{node{}} 
    						child{node{}}
    					} }
    child {  node{}  }
    child {  node{} [sibling distance=20mm]
                    child { node {}
                            child{node{}} 
    						child{node{}}}
                    child{node{}}
             } ;
\end{tikzpicture}
\caption{Non-isomorphic trees that maximise the Sombor index among all trees with given degree sequence $(3,3,3,3,2,1,1,1,1,1,1,1)$.}
\label{notunique}
\end{figure}
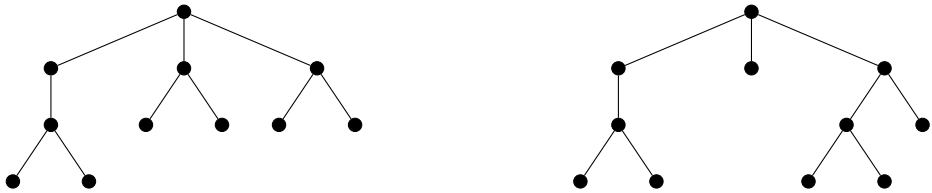

\subsection{Elements of $\mathbb{T}_D$ with minimum Sombor index}

A few recent papers \cite{damnjanovic2022, Wei2022} proved that the greedy tree $G(D)$ has the minimum Sombor index among all elements of $\mathbb{T}_{n,D}$:

\begin{thm}[\cite{damnjanovic2022, Wei2022}]\label{greedy}
Let $T \in \mathbb{T}_D$. Then, $\SO(T) \geq \SO(G(D))$, where $G(D)$ is the greedy tree corresponding to the sequence $D$.
\end{thm}
Since swapping branches between vertices of the same degree does not affect the Sombor index, there are cases where $G(D)$ is not the only tree that has the minimum Sombor index, that is for example if $D$ has repeated value other than $1$. We now prove that trees obtained from $G(D)$ by iteratively swapping branches between vertices of the same degree are the only ones that achieve the minimum Sombor index. Before proving the main theorem, 
we need one more lemma.

\begin{lem} \label{reverse2}
Let $T \in \mathbb{T}_D$. We consider $T$ to be rooted at an arbitrary vertex or an edge. Let $v_1,\dots,v_k$ be vertices at level $i$ and $u_m^1,\dots,u_m^{\deg(v_m)-1}$ the respective children of $v_m$ for $m\in\{1,\dots,k\}$. 
$T$ has the minimum Sombor index if it satisfies the following: If, for some $i,j\in\{1,\dots,k\}$, we have $ \de(v_i) > \de(v_j)$, then 
$$
\min_{j} \de({u_{i}}^j) \geq \max_{j}\de({u_{j}}^j).
$$
\end{lem}
\begin{proof}
The proof is similar to that of Lemma \ref{reverse}. Whenever the condition is not met we apply  the transformation in Lemma \ref{switching}, exchanging branches from $v_i$ and $v_j$, from $T$ to obtain a tree with same degree sequence but smaller Sombor index. 
\end{proof}

\begin{rem}
\label{Rem:SDegPL}
Note that a tree that satisfies Lemma \ref{reverse2} must have a vertex $u$ of degree $d_1$ with neighbors $v_2,\dots,v_{d_1+1}$ with $\deg_T(v_i)=d_i$ for all $i$. Furthermore, by swapping branches between vertices of degree $d_2$ if there are many (and not changing the Sombor index), $v_2$ can be chosen to have neighbours $v_{d_1+2},\dots,v_{d_1+d_2}$, again with $\deg_T(v_i)=d_i$.
\end{rem}

\begin{thm}\label{alter2}
Let $T \in \mathbb{T}_D$. $T$ is minimal if and only if it is isomorphic to $G(D)$ up to a  finite number of iterative exchanges of branches between vertices of the same degree.

\end{thm}
\begin{proof}
We prove that all trees of degree sequence $D$ that satisfies Lemma \ref{reverse2} has the same Sombor index, and hence have the minimum Sombor index among all element of $\mathbb{T}_D.$

Let $D=(d_1,\dots,d_t,1,\dots,1)$, with $d_t>1$. We use an induction on $t$. If $t=1$, the claim trivially holds, as there would be only one possible trees. Let us assume that it holds for some $t=j$ and let us consider a case where $t=j+1$. By Remark \ref{Rem:SDegPL}, after a finite number of iterative swapping branches between vertices of the same degree if needed, $T$ has a vertex $u$ of degree $d_1$ with neighbors $v_2,\dots,v_{d_1+1}$, and $v_2$ with neighbors $u,v_{d_1+2,\dots,v_{d_1+d_2}}$, where $\deg_T(v_i)=d_i$ for all $i$. Let $T'$ be the tree obtained from $T$ by merging $u$ and $v_2$ to form $w$.

$T'$ is a minimal tree among all trees of degree sequence $D'=(d_1+d_2-2,d_3,\dots,d_t,1,\dots,1)$. Otherwise, there would be a minimal tree $T''$ of degree sequence $D'$ and such that
\begin{equation}
\label{Eq:Induc}
\SO(T'')<\SO(T').
\end{equation} 
By Remark \ref{Rem:SDegPL} again, $T'$ has a vertex $v'_2$ with neighbors $v'_3,\dots,v'_{d_1+d_2+1}$ with $\deg_{T'}(v'_2)=d_1+d_2-2$ and $\deg_{T'}(v'_j)=d_j$ for all $j>2$. Let $T'''$ be obtained from $T''$ by breaking $v'_2$ into two adjacent vertices $v'_{1,1}$ and $v'_{1,2}$ such that $v'_{1,1}$ has degree $d_1$ and neighbors $v'_3,v'_4,\dots,v'_{d_1+1}$ in addition to $v'_2$, while $v'_{1,2}$ has degree $d_2$ and neighbors $v'_{d_1+2,\dots,d_1+d_2}$. By construction 
$$
\SO(T)-\SO(T')=\SO(T''')-\SO(T'').
$$
Combined with Equation \eqref{Eq:Induc}, this gives $\SO(T''')<\SO(T)$, which contradicts the assumption that $T$ is minimal.

Hence, there is a tree $K'$ obtained from $T'$ by iterative swapping of branches between vertices of the same degree such that there is an isomorphism
$f:V(K')\to V(G(D'))$. Necessarily, $f(u')$ is the only vertex of $G(D')$ with degree $d_1+d_2-2$. Let $K$ be obtained from $K$ by breaking the vertex of degree $d_1+d_2-2$ in $K'$ into two vertices of degree $d_1$ and $d_2$, respectively, in the same way as how $T$ is obtained from $T'$. Then $K$ can be obtained from $T$ after iterative swapping branches between vertices of the same degree.

The map
$$
\begin{array}{rcl}
  g:    V(K)&\to& V(G(D)) \\
     x& \mapsto& g(x)=
     \begin{cases}
     f(x) & \text{ if }x\in V(K')\\
     z_1 &\text{ if } x=u\\
     z_2& \text{ if } x=v_2.
     \end{cases}
\end{array}
$$
where $z_1$ and $z_2$ are vertices of $G(D)$ such that the multiset equality
$$
\{\deg_T(x):x\in N_T(u)\cup N_T(v_2)\}= \{\deg_{G(D)}(x):x\in N_{G(D)}(z_1)\cup N_{G(D)}(z_2)\},
$$
is an isomorphism.
\end{proof}
Lemma \ref{reverse2} also immediately imply the following theorem for rooted trees
\begin{thm}
\label{Thm:rooted}
For any tree $T$ with level degree sequence $D$, we have $\SO(T)\geq \SO(G(D))$.
\end{thm}
\begin{proof}
Apply Lemma \ref{reverse2} from the level furthest from the root
\end{proof}
\section{Trees of order $n$ with different degree sequences}
We denote by $S_n$ the set of all permutations of $\{1,\dots,n\}$. Let $A=(a_1,\dots,a_n)$ and $B=(b_1,\dots,b_n)$ be sequences of nonnegative numbers. We say that $B$ \textit{majorizes} $A$ if for all $1 \leq k \leq n$ we have
\[\sum_{i=1}^k a_i \leq \sum_{i=1}^k b_i.\]
If for any $\sigma \in S_n$ the sequence $B$  majorizes $(a_{\sigma(1)},\dots,a_{\sigma(n)})$, then we write
\[A \LE B.\]

We compare extremal trees associated to different degree sequences. This allows us to characterize the extremal trees with various degree conditions, relative to the Sombor index. It is well-known that two different degree sequences of trees can be recovered one step at a time as stated in the following lemma.
\begin{lem}[\cite{zhang2013}]
Let $D=(d_0,\dots,d_{n-1})$ and $D'=(d_0',\dots,d'_{n-1})$ be two nonincreasing degree sequences of trees. If $ D \LE D'$, then there exists a series of degree sequences $D_1,\dots,D_k$ such that $D \LE D_1 \LE \cdots \LE D_k \LE D'$, where $D_i$ and $D_{i+1}$ differ at exactly two entries, say $d_j$ (resp. $d'_j$) and $d_k$(resp. $d_k'$) of $D_i$ (resp. $D_{i+1}$), with $d'_j=d_j+1, d_k'=d_k-1$ and $j <k$.\label{lemzhang}
\end{lem}

Theorem \ref{diffdegree} is already presented in \cite{Wei2022}. We present it here for completeness, and we provide a slightly different proof.
\begin{thm}[\cite{Wei2022}]\label{diffdegree}
Let $D$ and $D'$ be the degree sequences of trees of the same order such that $D \LE D'$. Then we have
\[\SO(\Mi(D)) \leq \SO(\Mi(D')),\]
Equality holds if and only if $D=D'$.
\end{thm}

Before proving the main theorem, we will need a couple of Lemmas.

\begin{lem}[\cite{chen2022}]\label{lemChen}
Let $\phi(x, y) = \sqrt{x^2 + y^2} -\sqrt{(x - 1)^2 + y^2}$, where $x > 1$ and $y > 0$, then the function $\phi(x, y)$ is strictly increasing with respect to $x$ and strictly decreasing with respect to $y$.
\end{lem}

\begin{lem}\label{movingbig}
Let $\Mi(D)$ be the alternating greedy tree corresponding to the degree sequence $D$. Let $x,y$ be two vertices in $\Mi(D)$ such that $\de_{\Mi(D)}(x) \geq \de_{\Mi(D)}(y) \geq 2$. Consider that $\Mi(D)$ is rooted such that $x$ and $y$ are at the same level. Let $x'$ be a child of $y$. Let $T'=\Mi(D)-yx'+xx'$. Then
\begin{itemize}
    \item[1)] $T'\in \mathbb{T}_{D'}$, such that $D \LE D'$;
    \item[2)] $\SO(\Mi(D)) < \SO(T')$.
\end{itemize}
\end{lem}

\begin{proof}
The proof of 1) is straightforward. In fact, $\de_{\Mi(D)}(y)-1=\de_{T'}(y)$ and $\de_{\Mi(D)}(x)+1=\de_{T'}(x)$. Since $\de_{\Mi(D)}(x) \geq \de_{\Mi(D)}(y)$, we have $D \LE D'$.

Next, for 2), let $x^0$ and $y^0$ be the respective parents of $x$ and $y$. It is possible that $x^0=y^0$. Let $x^1_1,\dots,x^1_{\de_{\Mi(D)}(x)-1}$, and $y^1_1,\dots,y^1_{\de_{\Mi(D)}(y)-2}$ be the children of $x$ and $y$ different from $x'$. Since $\Mi(D)$ satisfies Lemma \ref{reverse}, we have (by swapping $x$ and $y$ if needed, when $\de_{\Mi(D)}(x^0) = \de_{\Mi(D)}(y_0)$)
\begin{align}
    &\de_{\Mi(D)}(x^0) \leq \de_{\Mi(D)}(y^0) \text{ and }\label{dpar}\\
    &\max_j \de_{\Mi(D)}(x^1_j) \leq \min_j \de_{\Mi(D)}(y^1_j)\label{dchild}.
\end{align} 

The contribution of the edges that do not contain $x$ nor $y$ as edge-point is not affected by this transformation. Let us compute the difference between the Sombor indices of the two trees. For ease of notation, we will denote by 
\begin{align*}
    &\de_{\Mi(D)}(x)=a, \quad \de_{\Mi(D)}(y)=b,   \de_{\Mi(D)}(x^0)=p_0,\quad \de_{\Mi(D)}(y^0)=p'_0,\\
    &\de_{\Mi(D)}(x^i_j)=c^i_j,\quad \de_{\Mi(D)}(y^i_j)=c'^i_j, \quad \de_{\Mi(D)}(x')=c .
\end{align*}

\begin{align}
    &\SO(T')-\SO(T)= \sqrt{(a+1)^2+p_0^2}+\sqrt{(b-1)^2+p_0'^2} +\sum_{j=1}^{a-1}\left(\sqrt{(a+1)^2+(c^1_{j})^2}\right)\nonumber\\
    &+\sum_{j=1}^{b-2}\left(\sqrt{(b-1)^2+(c'^1_j)^2}\right)+\sqrt{(a+1)^2+c^2}\nonumber\\
    &-\sqrt{a^2+p_0^2}-\sqrt{b^2+p_0'^2}-\sum_{j=1}^{a-1}\left(\sqrt{a^2+(c^1_{j})^2}\right)-\sum_{j=1}^{b-2}\left(\sqrt{b^2+(c'^1_j))^2}\right)-\sqrt{b^2+c^2}\nonumber\\
    \label{Eq:moving}
    &=\phi(a+1,p_0)-\phi(b,p_0')+\sum_{j=1}^{a-1}\phi(a+1,c^1_j)-\sum_{j=1}^{b-2}\phi(b,c'^1_j)\\ &\hspace{9cm}+\sqrt{(a+1)^2+c^2}-\sqrt{b^2+c^2}.\nonumber
\end{align}
Using Lemma \ref{lemChen} and conditions in \eqref{dpar}, we get
\begin{align*}\phi(a+1,p_0)>\phi(b,p_0)>\phi(b,p'_0).\end{align*}
Using again Lemma \ref{lemChen} and conditions in \eqref{dchild}, we have
\begin{align*}
   & \sum_{j=1}^{a-1}\phi(a+1,c^1_j)-\sum_{j=1}^{b-2}\phi(b,c'^1_j)\geq \sum_{j=1}^{b-2}\left(\phi(a+1,c^1_j)-\phi(b,c'^1_j)\right)>0.
   \end{align*}
Finally, it is clear that \[\sqrt{(a+1)^2+c^2}-\sqrt{b^2+c^2}>0.\]   
Combining all these results, we obtain $\SO(T')-\SO(T)>0$ as desired.
\end{proof} 

\begin{proof}[Proof of Theorem \ref{diffdegree}]
By Lemma \ref{lemzhang}, there exists a degree sequence $D_1$ with $D \LE D_1 \LE D'$ such that $D$ and $D_1$ only differ in two places, namely $D = (d_0, d_1,\dots, d_i,\dots, d_j,\dots, d_{n-1})$ and $D_1 = (d_0, d_1, \dots , d_{i-1},d_i + 1, d_{i+1} \dots ,d_{j-1}, d_j - 1,d_{j+1},\dots, d_{n-1})$ with $i <
j$ and $d_j \geq 2$. Let us consider two vertices $u, v$ in the alternating greedy tree $\Mi(D)$, such that
$\de_{\Mi(D)}(u) = d_i$ and $\de_{\Mi(D)}(v) = d_j$. 
Now, let $x$ be a child of $v$, and $T'=\Mi(D)-vx + ux$. Lemma \ref{movingbig} provides that $T' \in \mathbb{T}_{D_1}$, and
$\SO(T') > \SO(\Mi(D))$. Furthermore, by Theorem \ref{alter}, we have $\SO(T') \leq \SO(\Mi(D_1))$. Thus
\[\SO(\Mi(D_1)) \geq \SO(T') > \SO(\Mi(D)).\]
By iterating the same process, the theorem holds.
\end{proof}

It follows from the Theorems \ref{alter} and \ref{diffdegree} that $\SO(T)\leq \SO(\mathcal{M}(D))$ for any tree $T$ with degree sequence majorized by $D$. Hence the following corollaries.
\begin{cor}[\cite{gutman2021}]
For any $n$-vertex tree $T$, we have $\SO(T) \leq \SO(\mathcal{M}((n-1,1,\dots,1)).$ Equality only holds if the two graphs are isomorphic.
\end{cor}

\vspace{-0.2cm}

\begin{cor}
For any $n$-vertex tree $T$ with maximum degree $\Delta$, we have  
$$\SO(T)\leq \SO(\mathcal{M}(\Delta,\dots,\Delta,r+1,1,\dots,1)),$$ for some $0\leq r< d.$
\end{cor}

\vspace{-0.2cm}

\begin{cor}
For any $n$-vertex tree $T$ with exactly $\ell$ leaves, we have 
$$\SO(T) = \SO(\mathcal{M}(\ell,2,\dots,2,1\dots,1)).$$
\end{cor}

\vspace{-0.2cm}

\begin{cor}[\cite{li2022}]
Let $T$ be the $n$-vertex tree $T$ with diameter $d$. 
Then we have $$\SO(T )\leq \SO(\mathcal{M}(n-d+1,2,\dots,2,1,\dots,1)).$$ 
\end{cor}
We cannot expect a version of Theorem \ref{diffdegree} for the greedy trees, that is to get $\SO(G(D'))\leq \SO(G(D))$ whenever $D\LE D'$. Otherwise $\mathcal{M}(n-1,1,\dots,1)=G((n-1,1,\dots,1))$ would at the same time have the maximum and the minimum Sombor index among all trees of order $n$, which is false if $n>3$. In that case it is possible to have more $n$-vertex trees than the star with different Sombor indices. 

For the rest of this section, we compare $\SO(G(D'))$ and $\SO(G(D))$ given that $D\LE D'$. We discuss cases where  $\SO(G(D'))\geq \SO(G(D))$. 

An appropriate merging of any two pendent path from a tree reduces the Sombor index.
\begin{lem}
\label{Lem:WeakDiffDegGD}
Let $P_k$ be maximal pendent path of length $k$ attached to a vertex $u$ of degree $x\geq 3$ in a graph $G'$, and let $P_j$  be another maximal pendent path of length $j$ attached to a vertex $v$ of degree $y$ such that $x\geq y\geq 3$ in the same graph $G'$. Let $z$ be the endpoint of $P_j$ with degree $1$ and $w$ the neighbor of $u$ in $P_k$. Let $G=G'-uw+zw$. Then $\SO(G')>\SO(G)$.
\end{lem}
\begin{proof}
Let $a_1,\dots,a_{x-1}$ be the neighbors of $u$ other than $w$.

\textbf{Case 1:} Suppose that $k=j=1$. Then
\begin{align*}
\SO(G')-\SO(G)
&=\sum_{i=1}^{x-1}\left[\sqrt{x^2+\deg^2(a_i)}-\sqrt{(x-1)^2+\deg^2(a_i)}\right]+\sqrt{x^2+1}-\sqrt{4+1}\\
&+\sqrt{y^2+1}-\sqrt{y^2+4}
> \sqrt{x^2+1}-\sqrt{4+1}+\sqrt{y^2+1}-\sqrt{y^2+4}\\
&\geq f(y)=\sqrt{y^2+1}-\sqrt{4+1}+\sqrt{y^2+1}-\sqrt{y^2+4}>0,
\end{align*}
because $f'(y)=2y/\sqrt{y^2+1}-y/\sqrt{y^2+4}>0$ and $f(3)\approx 0.5>0.$

\textbf{Case 2:} Suppose that $k=1$ and $j\geq 2$. Then
\begin{align*}
\SO(G')-\SO(G)
&> \sqrt{x^2+1}-\sqrt{4+1}+\sqrt{4+1}-\sqrt{4+4}\\
&>\sqrt{9+1}-\sqrt{4+1}+\sqrt{4+1}-\sqrt{4+4}>0.
\end{align*}

\textbf{Case 3:} Suppose that $k\geq 2$ and $j=1$. Then
\begin{align*}
\SO(G')-\SO(G)
&> \sqrt{x^2+4}-\sqrt{4+4}+\sqrt{y^2+1}-\sqrt{y^2+4}\\
&>\sqrt{y^2+4}-\sqrt{4+4}+\sqrt{y^2+1}-\sqrt{y^2+4}
=-\sqrt{4+4}+\sqrt{y^2+1}>0.
\end{align*}

\textbf{Case 4:} Suppose that $k\geq 2$ and $j\geq 2$. Then
\begin{align*}
\SO(G')-\SO(G)
&> \sqrt{x^2+4}-\sqrt{4+4}+\sqrt{4+1}-\sqrt{4+4}\\
& \geq{9+4}-\sqrt{4+4}+\sqrt{4+1}-\sqrt{4+4}>0.1>0\\
\end{align*}
\end{proof}

\begin{rem}
In Lemma \ref{Lem:WeakDiffDegGD} if $D$ (resp. $D'$) is the degree sequence of $G$ (resp. $G'$), then we have $D\LE D'$. So, if $G'=G(D')$, then we have a case where $D\LE D'$ and $\SO(G(D'))=\SO(G')>\SO(G)\geq \SO(G(D)).$ 
\end{rem}
Iterative use of Lemme \ref{Lem:WeakDiffDegGD} provides a shorter alternative proof to the following finding from \cite{Wei2022}, on unicyclic graphs with given girth.
Let $T^k_n$ denote the unicyclic obtained by merging a vertex from a cycle of length $k$ to an end of a path of length $n-k+1$.
\begin{thm}[\cite{Wei2022}]
\label{Lem:Unicyclic}
Among all unicyclic graphs $U$ with girth $k$ and order $n$, we have 
\begin{align*}
\SO(U)\geq \SO(T^k_n).
\end{align*}
\end{thm}
\begin{proof}
As long as $U\neq T^k_n$, we can apply Lemma \ref{Lem:WeakDiffDegGD} to obtain a new unicyclic graph with girth $k$ order $n$ and smaller Sombor index. 
\end{proof}

Next, we study the effect of transferring more branches to a leaf. 
\begin{thm}
\label{Th:DiffDEgGreedy}
Let $u$ be a vertex of degree $x\geq 3$ and $v$ be a leaf of a rooted tree $T$. Let $T'$ be the tree obtained from $T$ by removing a branch $C$ from $u$ and attach it to $v$.Then,
$\SO(T)>\SO(T')$, 
\end{thm}

\begin{proof}
Let $c$ be the degree of the root of $C$.
%
%
As in \eqref{Eq:moving}, we have
\begin{align*}
\SO(T)-\SO(T')=\phi(x,a)-\phi(2,b)+\sum_{j=1}^{x-2}\phi(x,a_j)+\sqrt{x^2+c^2}-\sqrt{4+c^2}>0,
\end{align*}
since $x \geq 3$ and $b \geq a$.
\end{proof}

\begin{rem}
In Theorem \ref{Th:DiffDEgGreedy}, the degree sequence of $T$ majorises that of $T'$. In the special case where $T=G(D)$, $D'$ is the degree sequence of $T'$, we obtain a case where $D'\LE D$ and
$$
\SO(G(D))=\SO(T)>\SO(T')\geq \SO(G(D')).
$$
\end{rem}

Even though \cite{chen2022} already gave a description of a tree with given order and number of branching vertex, we provide a slightly stronger results as an immediate consequences of Theorem \ref{Th:DiffDEgGreedy}.
\begin{cor}
For any trees $T$ of order $n$ and at least $k$ branching vertices, we have
$$
\SO(T)\geq \SO(G(D)),
$$
where $D=(3,\dots,3,2,\dots,2,1\dots,1)$ and the entry $3$ repeats $k$ times.
\end{cor}
\begin{proof}
As long as there is a vertex of degree greater than $3$, by Theorem \ref{Th:DiffDEgGreedy}, we can move more branches from it to be attached to a leaf and obtain a tree that still have the same number of branching vertices and of which Sombor index is not increasing.

If the tree has more than $k$ branching vertices, we can again apply Theorem \ref{Th:DiffDEgGreedy} transforming a branching vertex into a vertex of degree $2$ by removing one branch from it and move it to a leaf.
\end{proof}

\bibliographystyle{abbrv} 
\bibliography{Sombor}
\end{document}